\documentclass[12pt]{amsart}

\usepackage{amssymb,latexsym,amsmath,amsthm,graphicx}
\usepackage{cite}

\newtheorem{thm}{Theorem}
\newtheorem{lemma}[thm]{Lemma}

\theoremstyle{remark}

\begin{document}
\author[A. Thomack]{Andrew Thomack}
\title[Random harmonic polynomials]{On the zeros of random harmonic polynomials: the naive model}

\begin{abstract}
A complex harmonic polynomial is the sum of a complex polynomial and a conjugated complex polynomial, of degrees $n$ and $m$ respectively.  Li and Wei present a formula for the expected number of zeros of a random harmonic polynomial in $\mathbb{C}$.  In this paper we prove that if $m$ is a fixed number, this expectation is asymptotically $n$ as $n\rightarrow\infty$, and if $m=n$ we find a lower and upper bound of order $n\log n$ for this expectation for sufficiently large $n$.
\end{abstract}

\maketitle

\section{Introduction}

A \emph{harmonic polynomial} is a function of the form $h(z)=p(z)+\overline{q(z)}$ where $p$ and $q$ are analytic polynomials so that $h$ is harmonic over $\mathbb{C}$ ~\cite{Small}.  We will refer to $h_{n,m}=p_n+\overline{q_m}$ as a harmonic polynomial of degree $n$ and $m$ where $p_n$ has degree $n$ and $q_m$ has degree $m$, and will use the convention $n\ge m$.  Since $h_{n,m}$ is not analytic, the Fundamental Theorem of Algebra does not tell us the number of zeros of $h_{n,m}$.  As with the number of real roots of real polynomials, there is a wide range of possibilities for the number of zeros of $h_{n,m}$, which we will denote as $\mathcal{N}_{n,m}$ or just $\mathcal{N}$.  Even if the number cannot be completely determined by $n$, some restrictions can yet be made on $\mathcal{N}_{n,m}$.  Wilmhurst proved $\mathcal{N}_{n,m}$ is bounded above by $\max\{n^2,m^2\}$ ~\cite{Wilmshurst} if certain degenerate cases of $m=n$ are excluded (occuring with probability $0$), while the lower bound is found to be $\max\{n,m\}$ by the generalized argument principle.  Both bounds are sharp, though for $m=1$, the upper bound has been improved (see ~\cite{Khavinson}) to $3n-2$, and it has been conjectured (see ~\cite{Wilmshurst}, ~\cite{Lee}, and ~\cite{Saez}) that there is a general improvement that is linear in $n$ for each fixed $m$.

Because of the range of values $\mathcal{N}_{n,m}$ could take, one might ask for the average value of $\mathcal{N}_{n,m}$ given a random harmonic polynomial, $h_{n,m}$.  A useful formula for determining the exact number of zeros of random real polynomials was provided by Kac ~\cite{Kac} while Dunnage provided a similar formula for trigonometric polynomials ~\cite{Dunnage}.  Shepp and Vanderbei go on to extend Kac's result to determining the number of zeros of complex polynomials for a given complex domain ~\cite{Vanderbei}.  

The Kac formula and its generalizations have been used by Li and Wei to give an explicit formula for $\mathbb{E}\mathcal{N}_{n,m}$ for a fixed $n$ and $m$ when the coefficients are independent Gaussian random variables.  In the case when the variances of the coefficients are scaled by a binomial factor, they also determined the asymptotics of their formula as $n\rightarrow\infty$ to be
\[\mathbb{E}\mathcal{N}_{n,m}\sim\left\{\begin{array}{lcl}\frac{\pi}{4}n^{3/2}&,& m=n\\n&,&m=\alpha n+o(n),\ \alpha\in[0,1)\end{array}\right.\]
An altered definition for the random coefficients, referred to as the ``truncated model'', was studied by Lerario and Lundberg who found different asymptotics for their model, $\mathbb{E}\mathcal{N}\sim c_\alpha n^{3/2}$ for ${m=\alpha n}$, $\alpha\in[0,1]$, $c_\alpha$ only dependent on $\alpha$ ~\cite{Lundberg}.

Although Li and Wei focused mainly on harmonic polynomials of the form
\[h_{n,m}(z)=\sum_{j=0}^na_jz^j+\sum_{j=0}^mb_j\overline z^j\] where $a_j$ and $b_j$ are independent Gaussians, $\mathbb{E}a_j=\mathbb{E}b_j=0$ with $\mathbb{E}a_j\overline a_k=\delta_{jk}\binom{n}{j}$ and $\mathbb{E}b_j\overline b_k=\delta_{jk}\binom{m}{j}$, they also considered what I will refer to as the ``naive'' model of coefficients of random polynomials, namely $\mathbb{E}a_j\overline a_k=\mathbb{E}b_j\overline b_k=\delta_{jk}$ so that the coefficients are i.i.d.  They proved the following ~\cite{LiWei}.
\begin{thm}\label{LiandWei}
The expected number of zeros of
\[h_{n,m}(z)=\sum_{j=0}^na_jz^j+\sum_{j=0}^mb_j\overline{z}^j\]on an open domain $T\subseteq\mathbb{C}$, denoted by $\mathbb{E}\mathcal{N}(T)$, is given by
\[\mathbb{E}\mathcal{N}(T)=\frac{1}{\pi}\int_T\frac{1}{\lvert z\rvert^2}\frac{r_1^2+r_2^2-2r_{12}^2}{r_3^2\sqrt{(r_1+r_2)^2-4r_{12}^2}}d\sigma(z)\]
where $\sigma(\cdot)$ denotes the Lebesgue measure on the complex plane and
\begin{align*}
&r_{12}=\left(\sum_{j=1}^nj\lvert z\rvert^{2j}\right)\left(\sum_{j=1}^mj\lvert z\rvert^{2j}\right),\qquad r_3=\sum_{j=0}^n\lvert z\rvert^{2j},
\\&r_1=r_3\sum_{j=0}^nj^2\lvert z\rvert^{2j}-\left(\sum_{j=0}^nj\lvert z\rvert^{2j}\right)^2,\quad r_2=r_3\sum_{j=0}^mj^2\lvert z\rvert^{2j}-\left(\sum_{j=0}^mj\lvert z\rvert^{2j}\right)^2.
\end{align*}
\end{thm}
Also in the publication by the same authors,
\begin{quote}
Numerical analysis suggests that the asymptotic of above expectation for $T=\mathbb{C}$ is $\lim_{n\rightarrow\infty}\mathbb{E}\mathcal{N}(\mathbb{C})/n=1$ for fixed $m$, but a rigorous analytic asymptotic hasn't been found.
\end{quote}
We will prove this analytically, that is,
\begin{thm}\label{myTheorem}
$\mathbb{E}\mathcal{N}(\mathbb{C})\sim n$ for fixed $m$.
\end{thm}
\noindent We will also explore the case when $n=m$.  We conjecture $\mathbb{E}\mathcal{N}(\mathbb{C})\sim c n\log n$ for some constant $c>0$, because we were able to find both an upper bound and lower bound of this order.
\begin{thm}\label{myTheorem2}
When $n=m$, there exists $c_2> c_1>0$ such that for $n$ sufficiently large\[c_1\le\frac{1}{n\log n}\mathbb{E}\mathcal{N}(\mathbb{C})\le c_2.\]
\end{thm}

\noindent {\bf Acknowledgment.}
I wish to thank Erik Lundberg and Antonio Lerario for leading me in the direction of this problem and discussing it with me at length.

\section{Preliminary Results}
The main tool in proving Theorem 2 will be the generalization of a common theorem in Real Analysis, the General Lebegue Dominated Convergence Theorem; see ~\cite{Royden} for proof.
\begin{thm}[General Lebegue Dominated Convergence Theorem]\label{GLDCT}
Let $\{f_n\}$ be a sequence of measurable functions on $E$ that converges pointwise almost everywhere on $E$ to $f$.  Suppose there is a sequence $\{g_n\}$ of nonnegative measurable functions on $E$ that converges pointwise almost everywhere on $E$ to $g$ and dominates $\{f_n\}$ on $E$ in the sense that $\lvert f_n\rvert\le g_n$ on $E$ for all $n$.  If \[\lim_{n\rightarrow\infty}\int_Eg_n=\int_Eg<\infty,\ \ then\ \ \lim_{n\rightarrow\infty}\int_Ef_n=\int_Ef.\]
\end{thm}
\noindent Thus in section \ref{MyThmSection} we search for the appropriate sequence $g_n$ in order to reach the result.

To maintain the clarity of the proof of the theorems, let us discuss a few of the inequalities used.
\begin{lemma}\label{lemma1}
If $x$ is a positive real number and $\alpha_k=\sum_{j=0}^kx^j$, $\beta_k=\sum_{j=1}^kjx^j$, and $\gamma_k=\sum_{j=1}^kj^2x^j$, then $(\alpha_n+\alpha_m)\gamma_n\gamma_m\ge \gamma_n\beta_m^2+\gamma_m\beta_n^2$, for $n,m\in\mathbb{N}$.
\end{lemma}
That $\alpha_k\gamma_k\ge\beta_k^2$ is a direct application of Cauchy-Schwarz, and the lemma follows.
\begin{lemma}\label{lemma2}
Let $\phi$ and $\psi$ be non-negative functions over $A\subseteq\mathbb{R}$ and let $f=\sqrt{\phi\psi}$.  Then $f\le\phi+\psi$ and $f$ is integrable over $A$ if $\phi$ and $\psi$ are.
\end{lemma}
\begin{proof}
We first notice that the arithmetic mean of $\phi$ and $\psi$ is $\frac{\phi+\psi}{2}$ while the geometric mean is $f$.  The arithmetic and geometric mean inequality says that $f\le\frac{\phi+\psi}{2}$, which we relax for the purposes of this lemma.  Since $f$ is a positive measurable function, integrability follows.
\end{proof}

\section{Proof of Theorem \ref{myTheorem}}\label{MyThmSection}
Within the context of Theorem \ref{LiandWei}, Theorem \ref{myTheorem} can be restated as follows
\begin{equation}\label{originalIntegral}
\lim_{n\rightarrow\infty}\frac{1}{n\pi}\int_{\mathbb{C}}\frac{1}{|z|^2}\frac{r_1^2+r_2^2-2r_{12}^2}{r_3^2\sqrt{(r_1+r_2)^2-4r_{12}^2}}d\sigma(z)=1.
\end{equation}
Beginning with the integral on the left side of (\ref{originalIntegral}), using polar coordinates and then a substitution of $\lvert z\rvert^2=1+\frac{t}{n}$ gives us
\[\lim_{n\rightarrow\infty}\int_{-n}^\infty\frac{1}{n(n+t)}\frac{r_1^2+r_2^2-2r_{12}^2}{r_3^2\sqrt{(r_1+r_2)^2-4r_{12}^2}}dt\qquad\text{where}\]
\[r_3=\sum_{j=0}^n\left(1+\textstyle\frac{t}{n}\right)^j+\sum_{j=0}^m\left(1+\textstyle\frac{t}{n}\right)^j\qquad r_{12}=\left(\sum_{j=1}^nj\left(1+\textstyle\frac{t}{n}\right)^j\right)\left(\sum_{j=1}^mj\left(1+\textstyle\frac{t}{n}\right)^j\right)\]
\[r_1=r_3\sum_{j=1}^nj^2\left(1+\textstyle\frac{t}{n}\right)^j-\left(\sum_{j=1}^nj\left(1+\textstyle\frac{t}{n}\right)^j\right)^2\] \[r_2=r_3\sum_{j=1}^mj^2\left(1+\textstyle\frac{t}{n}\right)^j-\left(\sum_{j=1}^mj\left(1+\textstyle\frac{t}{n}\right)^j\right)^2.\]
We use $\chi_n$ to be the function yielding 1 for values greater than $-n$ and 0 otherwise and, for a fixed $m$, define
\[f_{n}(t)=\chi_n(t)\frac{r_1^2+r_2^2-2r_{12}^2}{n(n+t)r_3^2\sqrt{(r_1+r_2)^2-4r_{12}^2}}.\]
We define the following functions of $t$ to simplify notation:
\[a_k=\sum_{j=0}^k\left(1+\frac{t}{n}\right)^j,\qquad b_k=\sum_{j=1}^kj\left(1+\frac{t}{n}\right)^j,\qquad c_k=\sum_{j=1}^kj^2\left(1+\frac{t}{n}\right)^j.\]

It follows from Lemma \ref{lemma1} and $r_3\ge0$ that
\[r_1r_2=r_3\left((a_n+a_m)c_nc_m-c_nb_m^2-c_mb_n^2\right)+r_{12}^2\ge r_{12}^2.\]
Using this, we can bound  $f_{n}$ 
\begin{align*}f_{n}(t)\le\frac{\sqrt{r_1^2+r_2^2-2r_{12}^2}}{n(n+t)r_3^2}&=\frac{\sqrt{(r_1-r_2)^2+2(r_1r_2-r_{12}^2)}}{n(n+t)r_3^2}\\&\le\frac{r_1-r_2}{n(n+t)r_3^2}+\sqrt{2}\frac{\sqrt{r_1r_2-r_{12}^2}}{n(n+t)r_3^2}.
\end{align*}
Then since
\[\frac{r_1r_2-r_{12}^2}{n(n+t)r_3^2}=\frac{c_m}{n^2(n+t)(a_n+a_m)}\cdot\frac{(a_n+a_m)c_n-b_n^2-\frac{b_m^2c_n}{c_m}}{(a_n+a_m)^2}\]and $b_nc_m\le b_mc_n$, by Lemma \ref{lemma2},
\[f_n(t)\le\chi_n(t)\left[\frac{r_1-r_2}{n(n+t)r_3^2}+\frac{2c_m}{n^2(1+\frac{t}{n})(a_n+a_m)}+\frac{(a_n+a_m)c_n-b_n^2-b_nb_m}{n(n+t)(a_n+a_m)^2}\right].\]
We can see $\frac{c_m}{1+t/n}<m^3$ when $t<0$ and $c_m<m^3(1+\frac{t}{n})^m$ when $t\ge 0$.  When $n>1$, we have $a_n+a_m\ge1+(1+\frac{t}{n})^2$, so we see that
\[f_n(t)\le\chi_n(t)\!\!\left[\!\frac{(a_n+a_m)(2c_n-c_m)-2b_n^2+b_m^2-b_nb_m}{n(n+t)(a_n+a_m)^2}+\left\{\begin{array}{ll}
\frac{2m^3}{n^2(1+(1+\frac{t}{n})^2)}&t<0
\\\frac{2m^3}{n^2(1+\frac{t}{n})^{n-m+1}}&t\ge0
\end{array}\right.\right]\!.\]The above bound for $f_n$ we will call $g_n$.

Now the integrability of $g_n$ is important if we wish to use Theorem \ref{GLDCT}.  We first note that $b_k=(n+t)\frac{d}{dt}[a_k]$ and $c_k=(n+t)\frac{d}{dt}[b_k]$.  From this we see
\[\frac{d}{dt}\left[\frac{2b_n-b_m}{a_n+a_m}\right]=\frac{(a_n+a_m)(2c_n-c_m)-2b_n^2-b_nb_m+b_m^2}{(n+t)(a_n+a_m)^2}.\]
Furthermore,
\[\frac{d}{dt}\left[\frac{2m^3}{n}\arctan\left(1+\textstyle{\frac{t}{n}}\right)\right]=\frac{2m^3}{2n^2+2nt+t^2}\]
and
\[\frac{d}{dt}\left[\frac{-2m^3}{n(n-m)\left(1+\textstyle{\frac{t}{n}}\right)^{n-m}}\right]=\frac{2m^3}{n^2\left(1+\textstyle{\frac{t}{n}}\right)^{n-m+1}}.\]Then $g_n$ has an antiderivative when $t\ne0$ or $-n$,
\[\int g_n=\frac{\chi_n}{n}\left[\frac{2b_n-b_m}{a_n+a_m}+2m^3\left\{\begin{array}{ll}\arctan\left(1+\textstyle{\frac{t}{n}}\right)&t<0\\\\\frac{-1}{(n-m)\left(1+\textstyle{\frac{t}{n}}\right)^{n-m}}&t>0
\end{array}\right.\right].\]
It follows that $\int_\mathbb{R}g_n=2+\frac{m^3\pi}{n}+\frac{2m^3}{n(n-m)}$ and clearly the limit as $n\rightarrow\infty$ of this integral is $2$.

Next we compute the point-wise limit of $g_n$ as $n\rightarrow\infty$, which we will call $g$.  Because $\sum_0^kx^j=\frac{x^{k+1}-1}{x-1}$ when $x\ne1$, then for $t\ne0$, $a_k=\frac{n}{t}\left(\left(1+\frac{t}{n}\right)^{k+1}-1\right)$.  Then
\begin{align*}&b_k=\textstyle{\frac{n+t}{t^2}}\left((kt-n)\left(1+\textstyle{\frac{t}{n}}\right)^k+n\right),\\\text{and}\quad &c_k=\textstyle\frac{n+t}{t^3}\left((k^2t^2-n(2k-1)t+2n^2)\left(1+\frac{t}{n}\right)^{k}-2n^2-nt\right).
\end{align*}
Then, to compute $g$, we can use the following limits,
\[\lim_{n\rightarrow\infty}\frac{a_n}{n+t}=\textstyle{\frac{1}{t}}(e^t-1)\quad\displaystyle\lim_{n\rightarrow\infty}\frac{b_n}{n(n+t)}=\textstyle{\frac{1}{t^2}}((t-1)e^t+1)\]
\[\lim_{n\rightarrow\infty}\frac{c_n}{n^2(n+t)}=\textstyle\frac{1}{t^3}\left((t^2-2t+2)e^t-2\right).\]
It is clear that for a fixed $t$ and $m$, the limits as $n$ increases to infinity of $a_m/(n+t)$, $b_m/(n(n+t))$ and $c_m/(n^2(n+t))$ are all $0$ as well as the limits of $\frac{2m^3}{n^2(1+(1+\frac{t}{n})^2)}$ and $\frac{2m^3}{n^2(1+\frac{t}{n})^{n-m+1}}$.  Then we can calculate
\[g(t)=2\frac{e^{2t}-(t^2+2)e^t+1}{t^2(e^t-1)^2}=2\frac{d}{dt}\left[\frac{(t-1)e^t+1}{t(e^t-1)}\right].\]
Meanwhile, if $f(t)$ is the point-wise limit of $f_n(t)$, then
\[f(t)=\frac{e^{2t}-(t^2+2)e^t+1)}{t^2(e^t-1)^2}=\frac{d}{dt}\left[\frac{(t-1)e^t+1}{t(e^t-1)}\right].\]
Then with a little computation we see $\int_\mathbb{R}g=2$ and $\int_\mathbb{R}f=1$.  Since
\[\lim_{n\rightarrow\infty}\int_{\mathbb{R}}g_n=\int_\mathbb{R}g\quad\text{then by Theorem \ref{GLDCT},}\quad\lim_{n\rightarrow\infty}\int_{\mathbb{R}}f_n=\int_\mathbb{R} f=1\]which proves Theorem \ref{myTheorem}.

\section{Proof of Theorem \ref{myTheorem2}}
We again start with the formula provided by Theorem \ref{LiandWei}, divide by the appropriate function of $n$ and attempt to take the limit.  This time we use the change of variables, $w=|z|^2$.  Then Theorem \ref{myTheorem2} is equivalent to
\begin{equation}\label{originalIntegral2}
\lim_{n\rightarrow\infty}\frac{1}{n\log n}\int_0^\infty\frac{1}{w}\frac{r_1^2+r_2^2-2r_{12}^2}{r_3^2\sqrt{(r_1+r_2)^2-4r_{12}^2}}dw\le1
\end{equation}
where $r_1=(a_n+a_m)c_n-b_n^2$, $r_2=(a_n+a_m)c_m-b_m^2$, $r_{12}=b_nb_m$ and $r_{3}=a_n+a_m$, again using the convention that \[a_k=\sum_{j=0}^kw^j,\quad b_k=\sum_{j=1}^kjw^j\quad\text{and}\quad c_k=\sum_{j=1}^kj^2w^j.\]  In the case where $n=m$, then $r_1=r_2=2a_nc_n-b_n^2$, $r_{12}=b_n^2$, and $r_3=2a_n$.  This allows us to simplify the integrand to
\[\frac{\sqrt{r_1^2-r_{12}^2}}{wr_3^2}=\frac{\sqrt{a_nc_n(a_nc_n-b_n^2)}}{2wa_n^2}\]

First we bound $\frac{c_n}{n^2a_n}$.  By increasing all the coefficients of $w^k$ in $c_n$ from $k^2$ to $n^2$, we can clearly see $\frac{c_n}{n^2a_n}$ is bounded above by $1$.  Furthermore,
\[\frac{c_n}{n^2}\ge\sum_{k=\lceil \frac{n}{2}\rceil}^n\frac{k^2}{n^2}w^k\ge\frac{1}{4}\sum_{k=\lceil \frac{n}{2}\rceil}^nw^k\]
and for $w\in(1,\infty)$,
\[a_n\le2\sum_{k=\lfloor \frac{n}{2}\rfloor}^nw^k\quad\implies\quad\frac{c_n}{n^2a_n}\ge\left\{\begin{array}{ll}
\frac{1}{8}&n\text{ even}\\\\\frac{1}{8}\left(1-1/a_{\frac{n+1}{2}}\right)&n\text{ odd}
\end{array}\right.\ge\frac{1}{16}.\]A better lower bound of $\frac{e-2}{e-1}$ can be found using algebraically complicated calculus but we simply require a lower bound greater than $0$.

Next we inspect another factor of the square of the integrand.  Similar to the case with fixed $m$, we can write $b_n$ as a product of $w$ and the derivative of $a_n$, $b_n=w\frac{d}{dw}[a_n]$ and similarly $c_n=w\frac{d}{dw}[b_n]$.  As before, we use the identity that $a_n=\frac{w^{n+1}-1}{w-1}$.  Then
\begin{align*}
\frac{a_nc_n-b_n^2}{wa_n^2}&=\frac{d}{dw}\left[\frac{b_n}{a_n}\right]=\frac{d}{dw}\left[w\frac{d}{dw}\left[\log a_n\right]\right]=\frac{d}{dw}\left[\frac{(n+1)w^{n+1}}{w^{n+1}-1}-\frac{w}{(w-1)}\right]\\&=\frac{d}{dw}\left[n+\frac{n+1}{w^{n+1}-1}-\frac{1}{w-1}\right]=\frac{1}{(w-1)^2}\left(1-\frac{(n+1)^2w^n}{a_n^2}\right).
\end{align*}
Here, $\frac{(n+1)^2w^n}{a_n^2}$ can be bounded below by $0$ for all $w>0$ and above for $w\in(1+\frac{1}{n},\infty)$ for $n$ greater than some fixed $N\in\mathbb{N}$.  The lower bound of zero is clear.  For the upper bound, we first note that it is less than $1$ for all $w>0$.  This can be seen by observing that $a_n/w^n=\sum_{k=0}^nw^{-k}$ and then by using the Cauchy Schwarz to see that
\[(n+1)^2=\left(\sum_{k=0}^nw^{-k/2}w^{k/2}\right)^2\le\left(\sum_{k=0}^nw^{-k}\right)\left(\sum_{k=0}^nw^k\right).\]Then we see that
\[\frac{d}{dw}\left[\frac{w^n}{a_n^2}\right]=\frac{w^{n-1}}{a_n^3}(na_n-2b_n)<0\]because
\begin{align*}
na_n-2b_n=&\sum_{k=0}^n(n-2k)w^k\\=&\sum_{k=0}^{\lfloor n/2\rfloor}(n-2k)w^k-\sum_{k=0}^{\lfloor(n-1)/2\rfloor}(n-2k)w^{n-k}\\=&\sum_{k=0}^{\lfloor(n-1)/2\rfloor}(n-2k)(w^k-w^{n-k})
\end{align*}and $w>1$.  Thus $1-\frac{(n+1)^2w^n}{a_n^2}$ is minimized at $w=1+\frac{1}{n}$.
\[1-\frac{(n+1)^2w^n}{a_n^2}\ge1-\frac{(n+1)^2(1+\frac{1}{n})^n}{n^2((1+\frac{1}{n})^{n+1}-1)^2}\]
The limit of this lower bound as $n\rightarrow\infty$ is $1-e(e-1)^{-2}\approx0.079$ and although this function seems to be decreasing as $n$ increases, we need not have a sharp lower bound and it will suffice that for large enough $n$, we have , say, half of this limit as a lower bound.  In summary, when $w\ge 1+\frac{1}{n}$
\begin{equation}
\frac{1}{16}\le\frac{c_n}{n^2a_n}\le1\qquad\text{and}\qquad \frac{1}{2}-\frac{e}{2(e-1)^2}\le1-\frac{(n+1)^2w^n}{a_n^2}\le1
\end{equation}moreover, the upper bounds apply for all $w>0$.

We will proceed by bounding the integrals over the following three intervals,
\[I=\left[0,1-\textstyle\frac{1}{n}\right),\quad J=\left[1-\textstyle\frac{1}{n},1+\frac{1}{n}\right],\quad\text{and}\quad K=\left(1+\textstyle\frac{1}{n},\infty\right).\]When $w\in I$,
\[0\le\frac{\sqrt{a_nc_n(a_nc_n-b_n^2)}}{2wa_n^2}\le\frac{1}{2\sqrt{w}(1-w)}\]
\[\implies0\le\int_I\frac{\sqrt{a_nc_n(a_nc_n-b_n^2)}}{2wa_n^2}\le\frac{1}{2}\left(1+\log\left(2+2\sqrt{1-\frac{1}{n}}-\frac{1}{n}\right)\right)\]
\[\implies0\le\lim_{n\rightarrow\infty}\frac{1}{\log n}\int_I\frac{\sqrt{a_nc_n(a_nc_n-b_n^2)}}{2wa_n^2}\le\frac{1}{2}.\]
Let $\alpha=1-e(e-1)^{-2}$.  Then in a similar way, for $w\in K$,
\[\frac{\sqrt{\alpha}}{8\sqrt{2}}\frac{1}{\sqrt{w}(w-1)}\le\frac{\sqrt{a_nc_n(a_nc_n-b_n^2)}}{2wa_n^2}\le\frac{1}{2\sqrt{w}(1-w)}\]
\[\implies0.025\approx\frac{\sqrt{\alpha}}{8\sqrt{2}}\le\lim_{n\rightarrow\infty}\frac{1}{\log n}\int_I\frac{\sqrt{a_nc_n(a_nc_n-b_n^2)}}{2wa_n^2}\le\frac{1}{2}\]
Finally, for $w\in J$,
\[\frac{a_nc_n-b_n^2}{4w^2a_n^2}\le\frac{c_n-b_n}{4w^2a_n}\le\frac{n^2-n}{4w^2}\le\frac{n^2}{4(n-1)}\]
so\[0\le\lim_{n\rightarrow\infty}\int_J\frac{\sqrt{a_nc_n-b_n^2}}{2wa_n\log n}dw\le\lim_{n\rightarrow\infty}\frac{n}{2\sqrt{n-1}\log n}\cdot\frac{2}{n}=0.\]
Thus,
\[\frac{\sqrt{\alpha}}{8\sqrt{2}}\le\lim_{n\rightarrow\infty}\frac{1}{n\log n}\int_0^\infty\frac{\sqrt{a_nc_n(a_nc_n-b_n^2)}}{2wa_n^2}dw\le1.\]

\section{Conclusion}
We have determined the asymptotic order of growth of $\mathbb{E}\mathcal{N}(\mathbb{C})$ in two cases, namely when $m$ is fixed and when $m$ grows with $n$.  The outcome in the former case, where there are on average asymptotically $n$ zeros, is not completely unexpected.  Indeed, one might expect that  as $n$ increases, $q_m(\bar{z})$ has a decreasing influence over the nature of $h_{n,m}$ when $m$ is fixed, so that $h_{n,m}$ may be heuristically treated as an analytic polynomial.  Our result gives evidence to this conclusion since on average $h_{n,m}$ has $n$ zeros.  On the other hand, when $m$ grows with $n$ we see a different behavior, implying that $h_{n,m}$ is quite distinguishable from an analytic polynomial of degree $n$.  Besides the additional $\log n$ factor in the total number of zeros, we also notice a disparity in the shape of the so-called ``first intensity'', that is, in the integrand appearing in Theorem \ref{LiandWei}.  In the case that $m$ is fixed the first intensity divided by $n$ converges to $0$ everywhere except along the unit circle.  Moreover, after the change of variable, we find that the area to the left of $t=0$ under the curve $f(t)$ is equal to the area to the right of $t=0$, where $t=0$ corresponds to $|z|=1$ through the change of variables.  This tells us that as $n$ increases, not only do the zeros concentrate on the unit circle, but they seem to do so in a symmetric manner (resembling the known symmetry of the first intensity of an analytic polynomial).  In contrast, the first intensity when $m=n$ accumulates to the unit circle in an asymmetric manner when divided by $n\log n$, which can be seen when dividing the integrand by $n$ instead of $\log n$ and taking a limit.  After the change of variables it is clear that this limit is $0$ for $w<1$ while it is $\frac{1}{2(w-1)\sqrt{w}}$ for $w>1$.

There are natural questions extending from this discussion.  Though we know the order of growth of $\mathbb{E}\mathcal{N}(\mathbb{C})$ for $n=m$ to be $n\log n$, it is yet an open problem to prove the limit of the quotient exists and, if it does, to obtain the explicit value of the constant.  One may also notice that we have only explored two cases for the behavior of $m$.  An obvious extension would be to investigate the case where $m=\alpha n$, $\alpha\in(0,1)$ as was done in ~\cite{LiWei} and ~\cite{Lundberg}.







\bibliographystyle{amsplain}

\bibliography{RandomPolyDraft}

\end{document}